\documentclass[11pt]{article}
\usepackage[top=2cm, bottom=2cm, left=2cm, right=2cm]{geometry}
\usepackage{tikz-cd}
\usepackage[alphabetic,lite]{amsrefs}                                          % for bibliography
\usepackage{amssymb}                                                           % for math symbols
\usepackage{amsthm}                                                            % for math theorems
\usepackage{appendix}                                                          % for appendices
\usepackage{array}                                                             % for tabulation
\usepackage{bm}                                                                % for bold letters in math
\usepackage{colonequals}                                                       % for better ":="
\usepackage{color}                                                             % for colorful text
\usepackage{enumitem}                                                          % for enumerating
\usepackage{graphicx}                                                          % for inserting pictures
\usepackage{indentfirst}                                                       % for first line indent
\usepackage{mathrsfs}                                                          % for "mathscr"
\usepackage{mathtools}                                                         % for better "amsmath"
\usepackage{moreenum}                                                          % for enumerating by Greek letters
\usepackage{stmaryrd}                                                          % for (at least) maps from
\usepackage{verbatim}                                                          % for "verbatim"
\usepackage[all,cmtip]{xy}                                                     % for commutative diagrams

\definecolor{tianred}{rgb}{0.79, 0.17, 0.57}                                   % for nice red
\definecolor{tianblue}{rgb}{0.0, 0.22, 0.66}                                   % for nice blue
\definecolor{tianpink}{rgb}{0.88, 0.56, 0.59}                                  % for nice pink
\definecolor{tiangreen}{rgb}{0.24, 0.82, 0.44}                                 % for nice green
\usepackage[colorlinks,
            linkcolor=tianred,
            anchorcolor=tianblue,
            citecolor=tiangreen,
            urlcolor=tianpink]{hyperref}                                       % for references
\usepackage{cleveref}                                                          % for better references

\usepackage[OT2,T1]{fontenc}
\DeclareSymbolFont{cyrletters}{OT2}{wncyr}{m}{n}
\DeclareMathSymbol{\RBe}{\mathalpha}{cyrletters}{"42}                          % Russian B (cyrillic letter)
\DeclareMathSymbol{\Che}{\mathalpha}{cyrletters}{"51}                          % Russian Y (cyrillic letter)
\DeclareMathSymbol{\Sha}{\mathalpha}{cyrletters}{"58}                          % Russian 山 (cyrillic letter)

\makeatletter
\DeclareRobustCommand\widecheck[1]{{\mathpalette\@widecheck{#1}}}
\def\@widecheck#1#2{%
    \setbox\z@\hbox{\m@th$#1#2$}%
    \setbox\tw@\hbox{\m@th$#1%
       \widehat{%
          \vrule\@width\z@\@height\ht\z@
          \vrule\@height\z@\@width\wd\z@}$}%
    \dp\tw@-\ht\z@
    \@tempdima\ht\z@ \advance\@tempdima2\ht\tw@ \divide\@tempdima\thr@@
    \setbox\tw@\hbox{%
       \raise\@tempdima\hbox{\scalebox{1}[-1]{\lower\@tempdima\box
\tw@}}}%
    {\ooalign{\box\tw@ \cr \box\z@}}}
\makeatother

\theoremstyle{plain}      \newtheorem{thm}{Theorem}[section]                   % 定理-英
\theoremstyle{plain}      \newtheorem{lem}[thm]{Lemma}                         % 引理-英
\theoremstyle{plain}                           % 推论-英
\theoremstyle{plain}      \newtheorem{prop}[thm]{Proposition}                  % 命题-英=法
\theoremstyle{plain}      \newtheorem{conjecture}[thm]{Conjecture}             % 猜想-英=法
\theoremstyle{definition} \newtheorem{rmk}[thm]{Remark}                        % 注释-英
\theoremstyle{definition} \newtheorem{df}[thm]{Definition}                     % 定义-英
\theoremstyle{definition} \newtheorem{eg}[thm]{Example}                        % 例子-英
\theoremstyle{definition}                        % 公理-英
\theoremstyle{definition} \newtheorem{hypo}[thm]{Hypothesis}                   % 假设-英
\theoremstyle{definition}                        % 习题-英
\theoremstyle{definition}                    % 小结-英
\theoremstyle{definition}                  % 符号-英=法
\theoremstyle{definition}          % 构造-英=法
\theoremstyle{definition}              % 约定-英=法
\theoremstyle{definition}                  % 问题-英=法
\theoremstyle{definition} \newtheorem{prop-def}[thm]{Proposition-Definition}    % 命题-定义

\theoremstyle{definition}
\newtheorem*{construction*}{Construction}                                      % 构造-不标号
\newtheorem*{conjecture*}{Conjecture}                                          % 猜想-不标号
\newtheorem*{hypothesis*}{Hypothesis}                                          % 假设-不标号
\newtheorem*{convention*}{Convention}                                          % 约定-不标号
\newtheorem*{notation*}{Notation}                                              % 符号-不标号
\newtheorem*{summary*}{Summary}                                                % 小结-不标号
\newtheorem*{qt*}{Question}                                                    % 问题-不标号
\newtheorem*{rmk*}{Remark}                                                     % 注释-不标号
\newtheorem*{fact*}{Fact}                                                      % 事实-不标号
\newtheorem*{lizi*}{Example}                                                   % 例子-不标号
\newtheorem*{def*}{Definition}                                                  % 定义-不标号
\theoremstyle{plain}
\newtheorem*{thm*}{Theorem}                                                    % 定理-不标号

\crefname{thm}{Theorem}{Theorems}                                              %
\crefname{thme}{Th\'eo\`eme}{Th\'eo\`emes}
\crefname{lem}{Lemma}{Lemmas}
\crefname{lemme}{Lemme}{Lemmes}
\crefname{eg}{Example}{Examples}
\crefname{ege}{Exemple}{Exemples}
\crefname{rmk}{Remark}{Remarks}
\crefname{rmke}{Remarque}{Remarques}
\crefname{cor}{Corollary}{Corollaries}
\crefname{core}{Corollaire}{Corollaires}
\crefname{def}{Definition}{Definitions}
\crefname{defe}{D\'efinition}{D\'efinitions}
\crefname{question}{Question}{Questions}
\crefname{prop}{Proposition}{Propositions}
\crefname{conjecture}{Conjecture}{Conjectures}

\newtheoremstyle{subsection-tweak}
{}{}{}{}{\bfseries}{}{.5em}{\thmnumber{(\@{#1}{}\@{#2}).}\thmnote{~{\bfseries#3.}}}
\theoremstyle{subsection-tweak}
\newtheorem{para}[thm]{}

\newcommand{\benum}{\begin{enumerate}[label={{\upshape(\alph*)}}]}             % (a),(b),(c),etc
\newcommand{\benuma}{\begin{enumerate}[label={{\upshape(\arabic*)}}]}          % (1),(2),(3),etc
\newcommand{\benumr}{\begin{enumerate}[label={{\upshape(\roman*)}}]}           % (i),(ii),(iii),etc
\newcommand{\eenum}{\end{enumerate}}
\newcommand{\bconj}{\begin{conjecture}}
\newcommand{\econj}{\end{conjecture}}
\newcommand{\bconjnn}{\begin{conjecture*}}
\newcommand{\econjnn}{\end{conjecture*}}
\newcommand{\begs}{\begin{eg}\hfill\benuma}                                    % examples
\newcommand{\eegs}{\eenum\end{eg}}                                      

\newcommand{\brmks}{\begin{rmk}\hfill\benuma}                                  % remarks
\newcommand{\ermks}{\eenum\end{rmk}}                                           % remarks
\newcommand{\bitem}{\begin{itemize}}                                           % itemize
\newcommand{\eitem}{\end{itemize}}                                             % itemize
\newcommand{\be}{\begin{equation}}                                             % equation
\newcommand{\ee}{\end{equation}}                                               % equation
\newcommand{\benn}{\begin{equation*}}                                          % non-numbering
\newcommand{\eenn}{\end{equation*}}                                            % non-numbering
\newcommand{\bqt}{\begin{qt*}\rm}                                              % question
\newcommand{\eqt}{\end{qt*}}                                                   % question
\newcommand{\bqtr}{\begin{qt*}\rm\coLR}                                        % question in red color
\newcommand{\eqtr}{\end{qt*}}                                                  % question in red color
\newcommand{\beac}{\begin{equation}\begin{array}{c}}                           % diagram central numbering
\newcommand{\eeac}{\end{array}\end{equation}}                                  % diagram central numbering
\newcommand{\beqn}{\begin{eqnarray*}}
\newcommand{\eeqn}{\end{eqnarray*}}
\newcommand{\bdf}{\begin{df}}
\newcommand{\bdfhf}{\begin{df}\hfill}
\newcommand{\edf}{\end{df}}
\newcommand{\brmk}{\begin{rmk}}
\newcommand{\brmkhf}{\begin{rmk}\hfill}
\newcommand{\ermk}{\end{rmk}}

\newcommand{\mrm}{\mathrm}

   \newcommand{\BX}{\mathbf{X}}

\newcommand{\BBG}{\mathbb{G}}

\newcommand{\ol}{\overline}                                                    % overline
\newcommand{\ce}{\colonequals}                                                 % colon + equal
\newcommand{\uun}{^{(1)}}                                                      % upper un

\DeclareMathOperator{\gal}{Gal}                                                % Galois group

\DeclareMathOperator{\Image}{Im}                                               % Image
\renewcommand{\Im}{\Image}                                                     % Image, imaginary part
\DeclareMathOperator{\Ker}{Ker}                                                % Kernel
\DeclareMathOperator{\cone}{cone}                                              % Mapping cone
\DeclareMathOperator{\GL}{\mathbf{GL}}                                         % general linear group
\DeclareMathOperator{\SL}{\mathbf{SL}}                                         % special linear group
\newcommand{\Gss}{G^{\mathrm{ss}}}                                             % semi-simple subgroup of G
\newcommand{\Gsc}{G^{\mathrm{sc}}}                                             % simply connected covering of G^ss
\newcommand{\Tsc}{T^{\mathrm{sc}}}                                             % maximal torus of G^sc
\DeclareMathOperator{\Res}{Res}                                                % restriction
\newcommand{\wrt}{with respect to~}

\begin{document}
\title{\textbf{On Weak Approximation of Reductive Groups over Higher Dimensional Function Fields}}
\author{Zhongda LI, Che LIU, Haoxiang PAN}
\maketitle

\begin{abstract}
   Let $k$ be a $d$-local field of characteristic 0, and let $K$ be the function field of a nice curve over $k$. We give a defect to weak approximation for reductive groups over $K$ using arithmetic dualities.

\end{abstract}

\section{Introduction}

The present article aims to describe the defect to 
weak approximation for reductive groups over the function field of a 
curve over some $d$-local field (see \ref{dlocal} for its definition). 
This work generalizes the previous works of 
Harari--Scheiderer--Szamuely \cite{HSS15} 
and Izquierdo \cite{Izq16these} on tori over these function fields, 
and the work of Tian \cite{Tia21WA} on reductive groups over $p$-adic function fields.
Our strategy is to attach each connected reductive group $G$ with
a two-term short complex $C$ of tori as in \cites{Bor98, Dem11, Tia21WA}.
Subsequently, arithmetic dualities enable us to deduce an exact sequence of groups which describes the desired defect to weak approximation.

Let $K$ be the function field of a nice curve $X$ defined over some $d$-local field.
Let $G$ be a connected reductive group over $K$.
However, in extending these works to $G$, 
it turns out that $\Sha^2(\BBG_m)\not=0$ could happen when $d\ge 2$.
As mentioned in \cite{Izq16these}*{Chapitre II, \S1},
we need the following technical assumption:
%making the situation very different from the $d=1$ case. So before going further, 

\begin{hypo}\label{hypo:Sha2Vanish}
Let $T\subset G$ be a fixed maximal $K$-torus and 
let $L|K$ be a fixed finite Galois extension that splits $T$.
Assume $\Sha^2(L,\BBG_{m})=0$. According to \cite{Izq16these}*{Chapitre~1, Lemme~3.20}, the vanishing $\Sha^2(L,\BBG_{m})=0$ is equivalent to $\Sha^{d+2}(L,\mathbb{Z}(d))=0$. 
\end{hypo}

We point out that $\Sha^2(\BBG_{m})=0$ automatically holds for any number field by global class field theory.
Moreover, when $d=1$, 
one even has $\Sha^2_{\omega}(T)=0$ for any torus $T$ 
by \cite{HSS15}*{Lemma 4.2}.
These phenomena suggest that the assumption $\Sha^2(L,\BBG_{m})=0$ is quite natural.
For more detailed discussion on the nullity of $\Sha^2(\mathbb{G}_m)$, see \cite{Izq16these}*{Chapitre~1, Section~1}. 
Now let $\tilde{C}$ be the dual of $C$ (see (\ref{ShortComplex}) below).
With \Cref{hypo:Sha2Vanish}, we can control the finiteness of $\Sha^{d+1}_S(K,\tilde{C})$ and further establish a map $\Sha_\omega^{d+1}(K,\tilde{C})^D \rightarrow \Sha^1(K,C)$.
More precisely, we have

\begin{thm}[Global duality]
 There is a perfect pairing of finite groups
	$$\Sha^1(K, C) \times \Sha^{d+1}(K, \tilde{C}) \to \mathbb{Q} / \mathbb{Z}.$$
\end{thm}

Notably, we replace $\overline{\Sha^1(K,C)}$ in \cite{Izq16these}*{Chapitre~1, Corollarie~4.18} by $\Sha^1(K,C)$. 
Together with the local duality, we manage to construct a sequence of groups
%Then we prove the exact sequence for a finite set of places $S$, i.e., 
$$
      1\rightarrow {\overline{G(K)}}_S \rightarrow \prod_{v\in S}G(K_v) \rightarrow \Sha_S^{d+1}(K,\tilde{C})^D \rightarrow \Sha^1(K,C)\rightarrow 0.
$$
In proving the exactness of this sequence,  
the main difficulty lies in the exactness at $\Sha_S^{d+1}(K,\tilde{C})^D$. 
To overcome this, we introduce the following hypothesis (which also automatically holds for semi-simple simply connected groups over number fields by \cite{PR94}*{Theorem~7.8}).

\begin{hypo}\label{hypo:GscWA}
    Let $\Gsc$ be the simply connected cover of the derived subgroup of $G$.
    We make an assumption that $G^{sc}$ satisfies the weak approximation property and has a quasi-trivial maximal torus.
\end{hypo}

Now we state the main theorem of this article.

\begin{thm}\label{MainThm}
Let $G$ be a connected reductive group over $K$ satisfying 
\textup{Hypothesis \ref{hypo:Sha2Vanish} and \ref{hypo:GscWA}}. Let $X^{(1)}$ denotes the set of all closed points of $X$. 
\begin{enumerate} 
    \item Let $S\subset X^{(1)}$ be a finite set of 
    places. There is an exact sequence
        \begin{equation} \label{Seq:FiniteMainTheorem}
      1 \rightarrow \overline{G(K)}_S \rightarrow \prod_{v\in S}G(K_v) \rightarrow \Sha_S^{d+1}(K,\tilde{C})^D \rightarrow \Sha^1(K,C) \rightarrow 0.  
    \end{equation}
    Here $\overline{G(K)}_S$ denotes the closure of 
    the diagonal image of $G(K)$ in $\prod_{v\in S}
    G(K_v)$ for the product topology.
    \item There is an exact sequence
         \begin{equation*}
      1 \rightarrow \overline{G(K)} \rightarrow \prod_{v\in X^{(1)}}G(K_v) \rightarrow \Sha_{\omega}^{d+1}(K,\tilde{C})^D \rightarrow \Sha^1(K,C) \rightarrow 0.  
    \end{equation*}
    Here $\overline{G(K)}$ denotes the closure 
    of the diagonal image of $G(K)$ in $\prod_{v\in X^{(1)}}
    G(K_v)$ for the product topology.
\end{enumerate}

\end{thm}

Under \Cref{hypo:GscWA},
we may assume that $G$ admits a special covering (see (\ref{para: special covering}) below) 
which enables us to prove the exactness through diagram chasing.  Finally, we conclude the proof by taking projective limits.
Besides, in the proof, we will see that $\overline{G(K)}$ is a normal subgroup of $\prod_{v\in X^{(1)}}G(K_v)$.

\textbf{Acknowledgments.} This work was done under the supervision of Yisheng Tian. We would like to 
express our deepest gratitude to him for his invaluable guidance throughout the entire process of this research. 
We also thank Academy of Mathematics and Systems 
Science, CAS for providing an excellent environment. 
We appreciate the Algebra and Number Theory Summer School for its organization and financial support.

%%%%%%%%%%%%%%%%%%%%%%%%%%%%%%%%%%%%%%%%%%%%%%%%%%%%%%%%%%%%%%%
\section{Notation and Conventions}
%d-locol field
    \begin{para}[$d$-local fields]\label{dlocal}
        A $0$-local field is a finite field. 
        For an integer $d>0$, a $d$-local field is a complete discrete valuation field whose residue field is a $(d-1)$-local field.
        Thus $1$-local fields are just local fields in the usual sense.
    \end{para}
    
    \begin{para}[Function fields]
    Throughout, let $X$ be a smooth projective geometrically integral curve over some $d$-local field $k$ of characteristic 0.
    Let $K=k(X)$ be the function field of $X$.
    For any closed point $v\in X\uun$, the local ring 
    $\mathscr{O}_{X,v}$ is a $1$-dimensional regular local ring, hence it is a discrete valuation ring.
    Thus we write $K_v$ for the completion of $K$ \wrt $v$
    which is a $(d+1)$-local field.
    Let $\ol{K}$ be a fixed algebraic closure of $K$ and 
    let $\Gamma\ce \gal(\ol{K}|K)$ be the absolute Galois group of $K$.
    \end{para}
    
    \begin{para}[Algebraic groups] \label{AlgebraicGroups}
    
    Given a linear algebraic group $G$ over $K$, there exists a largest connected smooth normal solvable (resp. unipotent) subgroup, denoted by $R(G)$ (resp. $R_u(G)$). If $R(G_{\overline{K}})$ (resp. $R_u(G_{\overline{K}})$) is trivial, then $G$ is called semisimple (resp. reductive).

    % G^ss and G^sc
        Let $G$ be a connected reductive group over $K$. Let  $G^{\mathrm{ss}}$ be its derived subgroup, which is semisimple. 
        Let $G^{\mathrm{sc}}$ be the simply connected cover of $G^{\mathrm{ss}}$.
        So there is a central isogeny $\Gsc\to \Gss$.
        Throughout, let $\rho:\Gsc\to \Gss\to G$ be the composite 
        (see \cite{Mil17}*{Theorem~18.25} for details). 
        For example, if $G=\GL_n$, then we have $\Gss=\Gsc=\SL_n$.
    \end{para}
    
    \begin{para}[Tori]
        An algebraic group $G$ over $K$ is a torus if 
        $G_L\cong \BBG^{\oplus d}_{m,L}$ for some finite 
        separable extension $L/K$.
        A $K$-torus $T$ is quasi-trivial if 
        $T\cong \prod\Res_{L_i/K} \BBG_{m}$ for some 
        finite separable extensions $L_i/K$.
        For a $K$-torus $T$, let 
        $\mathbf{X}^{*}(T)$ (resp. $\mathbf{X}_*(T)$) be its module of characters (resp. cocharacters). 
        
        Now suppose that $G$ is a connected reductive group. 
        For a maximal torus $T$ of $G$,
        the inverse image $\Tsc \coloneqq \rho^{-1}(T)$
        is a maximal torus of $\Gsc$. See \cite{SGA3II}*{Exposé~XI, Théoème~7.1(f), Exposé~XIV, Proposition~4.9} and \cite{SGA3III}*{Exposé~XXIl, Proposition~6.2.8}.
    \end{para}

\begin{para}
[Special covering]\label{para: special covering}
An isogeny $G_0 \to G$ of connected reductive groups is called a special covering if $G_0$ is the product of a semi-simple simply connected group and a quasi-trivial torus. 
\end{para}

\begin{para}
[Motivic complexes]
By a $K$-variety $Y$, we always mean a separated $K$-scheme of finite type.
Bloch defined the cycle complex $z^i(Y,\bullet)$ in \cite{Blo86}. 
When $Y$ is smooth, we denote the \'etale motivic complex over $Y$
(which is a complex of \'etale sheaves) 
by $\mathbb{Z}(i)=z^i(-,\bullet)[-2i]$ on the small \'etale site of $Y$. As an example, we have quasi-isomorphisms 
$\mathbb{Z}(0)\cong\mathbb{Z}$ and $\mathbb{Z}(1)[1]\cong\mathbb{G}_m$ by \cite{Blo86}*{Corollary~6.4}. 
\end{para}    

\begin{para}
[Short complex] \label{ShortComplex}
Let $G$ be a connected reductive group and choose a maximal torus $T$ of $G$. Denote $T^\mathrm{sc}$ the corresponding maximal torus of $T$ in $G^\mathrm{sc}$ and $\rho_T: T^\mathrm{sc}\rightarrow T$ the morphism induced by $\rho\colon G^\mathrm{sc}\rightarrow G$. Define $\tilde{T}=\mathbf{X}^*(T)\otimes^\mathbf{L}\mathbb{Z}(d)$ and $\tilde{T}^\mathrm{sc}=\mathbf{X}^*(\Tsc)\otimes^\mathbf{L}\mathbb{Z}(d)$ and $\tilde{\rho_T}: \tilde{T}\rightarrow\tilde{T}^\mathrm{sc}$ the morphism induced by $\rho_T$.  Denote $C=\cone[\rho_T]=[T^\mathrm{sc}\rightarrow T]$ and $\tilde{C}=\cone[\tilde{\rho}_T]=[\tilde{T}\rightarrow\tilde{T}^\mathrm{sc}]$, which are concentrated in degree $-1$ and $0$. 
\end{para}

\begin{para}
[Tate--Shafarevich groups] 
For a subset $S\subseteq X^{(1)}$, let $\Sha^{i}_S(K,C)$ be the Tate--Shafarevich group associated to the Galois hypercohomology of $C$: 
\[
\Sha^{i}_S(K,C)=\Ker\left(\mathbb{H}^i(K,C)\rightarrow\prod_{v\notin S}\mathbb{H}^i(K_v,C)\right).
\]
We put further $\Sha^i(K,C)=\Sha^i_\varnothing(K,C)$
and $\Sha^i_{\omega}(K,C)=\bigcup_S \Sha^i_S(K,C)$ where 
$S$ runs through all finite subsets $S\subset X^{(1)}$.
\end{para}

\begin{para}
[Weak approximation]
Let $L$ be  a topological field (for example, a local field). 
We equip $\mathbb{A}^n(L)=L^n$ with the product topology. 
If $Y$ is a closed subvariety of $\mathbb{A}^n$, 
we equip $Y(L)\subset\mathbb{A}^n(L)$ with the subspace topology. 
In general, 
an $L$-variety $Y$ is obtained by gluing affine open subsets, 
so we use the same gluing data to glue the topological spaces. 
%Two different affine open coverings give the same topology on $X(L)$, as one can check by comparison with a common refinement. 
In particular, we give $\prod_{v\in X^{(1)}}Y(K_v)$ the product topology.
We say that $Y$ satisfies weak approximation (WA) if 
 \;$\overline{Y(K)}=\prod_{v\in X^{(1)}}Y(K_v)$. 
% Here $Y(K)$ is viewed as a subset of  $\prod_{v\in X^{(1)}}Y(K_v)$ and give it the subspace topology.
\end{para}

\section{Arithmetic dualities}
To construct the arrows in the sequence 
(\ref{Seq:FiniteMainTheorem}), we need to introduce 
arithmetic dualities. We begin with local dualities 
which enables us to link $\prod_{v\in S}G(K_v)$ with 
$\Sha_S^{d+1}(K,\tilde{C})$.
The next proposition will be used to establish the duality map and deduce certain finiteness of Tate--Shafarevich groups.

\begin{prop}\label{prop:MotivicComplex}
    For any field $L$ with character $0$, we have $\mathbb{H}^{d+1}(L,\mathbb{Z}(d))=0$ and $\mathbb{H}^{d+2}(L,\mathbb{Q}/\mathbb{Z}(d+1))\cong\mathbb{H}^{d+3}(L,\mathbb{Z}(d+1))$.
\end{prop}

\begin{proof}
    The first statement follows from \cite{Izq16these}*{Chapitre~1, Th\'eor\`eme~0.4(ii)(vii)}. Since 
    $$0\to \mathbb{Z}(i)\to \mathbb{Q}(i)\to\mathbb{Q}/\mathbb{Z}(i)\to 0$$
    is exact for any non-negative integer $i$, the second statement follows from \cite{Izq16these}*{Chapitre~1, Th\'eor\`eme~0.4(iii)}.
\end{proof}

\begin{thm} [Local Duality]\label{LocalDuality}
     There are perfect pairings of abelian groups:
    $$\mathbb{H}^0(K_v, C)^{\wedge} \times \mathbb{H}^{d+1}(K_v, \tilde{C}) \to \mathbb{Q} / \mathbb{Z}$$
\end{thm}
\begin{proof}
    The pairing is given by the composition of cup product and \cite{Izq16these}*{Chapitre~1, Lemme~4.3}: 
    \[
        \mathbb{H}^0(K_v, C)^{\wedge} \times \mathbb{H}^{d+1}(K_v, \tilde{C}) \to \mathbb{H}^{d+1}(K_v,\mathbb{Z}(d+1)[2])\cong\mathbb{H}^{d+2}(K_v,\mathbb{Q}/\mathbb{Z}(d+1))\cong\mathbb{Q}/\mathbb{Z}. 
    \]
    See \cite{Izq16these}*{Chapitre~1, Proposition~4.9} for a proof of the perfectness.
\end{proof}

The rest of this section is devoted to global dualities. We state the following lemmas as prerequisites. 

\begin{lem} \label{Nullities}
If $R$ is a quasi-trivial torus over $K$ and $\tilde{R}=\operatorname{Hom}(R,\mathbb{G}_m)\otimes^\mathbf{L}\mathbb{Z}(d)$, then:
\begin{enumerate}
    \item[\textup{(i)}] \label{Cohomology1VanishForQuasiTrivial} $H^1(K,R)=0$ and $H^1(K_v,R)=0$ for all $v\in X^{(1)}$. 
    \item[\textup{(ii)}] $\mathbb{H}^{d+1}(K,\tilde{R})=0$ and $\mathbb{H}^{d+1}(K_v,\tilde{R})=0$ for all $v\in X^{(1)}$. 
\end{enumerate}
\end{lem}

\begin{proof}
   % For a field $k$, let $\Bar{k}$ denote the separable closure of $k$ and let $\Gamma_k=\gal(\Bar{k}/k)$. 
    %Without loss of generality, we may suppose $R=\mathrm{Res}_{F/K}(\mathbb{G}_{m,F})$ for some finite separable extension $F/K$, where $\mathrm{Res}_{F/K}$ denotes the Weil scalar restriction from $F$ to $K$. 
    After Shapiro's lemma, (i)
    is a direct consequence of Hilbert's Theorem 90 and (ii) follows from \cref{prop:MotivicComplex}(i).
\end{proof}

\begin{lem} \label{ShaFinite}
The Tate--Shafarevich group $\Sha^{d+1}_S(K,\tilde{C})$ is finite.
\end{lem}
\begin{proof}
We first prove that $\Sha^{d+1}_S(\tilde{C})$ has finite exponent. 
From the distinguished triangles $\tilde{T} \to \tilde{\Tsc} \to \tilde{C} \to \tilde{T}[1]$, one can obtain the following commutative diagram of abelian groups:
\begin{center}
    \begin{tikzcd}
    H^{d+1}(K,\widetilde{\Tsc}) \arrow[r] \arrow{d} & H^{d+1}(K,\tilde{C}) \arrow{r} \arrow{d} & H^{d+2}(K,\tilde{T}) \arrow{r} \arrow{d} &
    H^{d+2}(K,\widetilde{\Tsc}) \arrow{d} \\
    \displaystyle\prod_{v \not\in S} H^{d+1}(K,\widetilde{\Tsc}) \arrow[r]  & \displaystyle\prod_{v \not\in S} H^{d+1}(K,\tilde{C}) \arrow{r} & \displaystyle\displaystyle\prod_{v \not\in S} H^{d+2}(K,\tilde{T}) \arrow{r}  &
    \displaystyle\prod_{v \not\in S} H^{d+2}(K,\widetilde{\Tsc})
    \end{tikzcd}
\end{center}
where each row is exact. By \cref{Cohomology1VanishForQuasiTrivial}(ii), $H^{d+1}(K,\widetilde{\Tsc}) = H^{d+1}(K_v,\widetilde{\Tsc}) = 0$, so we have an exact sequence of groups:
\begin{equation} \label{ShaGroupInjective}
  0 \to \Sha^{d+1}_S(\tilde{C}) \to \Sha^{d+2}_S(\tilde{T}).  
\end{equation}
Recall Hypothesis \ref{hypo:Sha2Vanish} that $L / K$ is a fixed finite Galois extension such that $T$ is split over $L$. Consider the following commutative diagram, where $S^\prime \coloneqq \{ w : w \mid v \text{ for some } v \in S \}$: 
% \begin{center}
% \begin{tikzcd}
% H^{d+2}(K,\tilde{T}) \arrow{r} \arrow[d, "\operatorname{res}"] & \displaystyle\prod_{v \in S} H^{d+2}(K_v, \tilde{T}) \arrow{d} \\
% H^{d+2}(L, \tilde{T}) \arrow{r} \arrow[d, "\operatorname{cor}"] &
% \displaystyle\prod_{w\in S^\prime}H^{d+2}(L_w, \tilde{T}). \\
% H^{d+2}(K,\tilde{T})
% \end{tikzcd}
% \end{center}
\begin{center}
    \begin{tikzcd}
    H^{d+2}(K,\tilde{T}) \arrow[d] \arrow[r, "\operatorname{res}"] &
    H^{d+2}(L,\tilde{T}) \arrow[d] \arrow[r, "\operatorname{cor}"] &
    H^{d+2}(K,\tilde{T}) \\
    \prod_{v \in S} H^{d+2}(K_v, \tilde{T}) \arrow[r] &
    \prod_{w \in S'} H^{d+2}(L_w,\tilde{T}).
    \end{tikzcd}
\end{center}
Taking the kernels of the two horizontal arrows yields maps
\[
\Sha^{d+2}_S(K,\tilde{T}) \stackrel{\operatorname{res}}{\longrightarrow} \Sha^{d+2}_{S'}(L,\tilde{T}) \stackrel{\operatorname{cor}}{\longrightarrow}
H^{d+2}(K,\tilde{T}).
\]
By our Hypothesis \ref{hypo:Sha2Vanish}, we obtain
\[
\Sha^{d+2}_{S'}(L,\tilde{T}) \cong \Sha^{d+2}_{S^\prime}(L, \mathbb{Z}(d))^m = 0, 
\]
but $\operatorname{cor} \circ \operatorname{res} = [L:K] \cdot \operatorname{Id}$, so we conclude that $\Sha^{d+2}_S(K,\tilde{T})$ is of finite exponent. Then the injection (\ref{ShaGroupInjective}) implies that the group $\Sha^{d+1}_S(K, \tilde{C})$ has finite exponent as well.

Subsequently, we turn to the following exact sequence:
\[
0 \to \Sha^{d+1}(K, \tilde{C}) \to \Sha^{d+1}_S(K, \tilde{C}) \to \bigoplus_{v \in S} H^{d+1}(K_v, \tilde{C}).
\]
We obtain the following injection employing the same method as in (\ref{ShaGroupInjective})

\[
0 \to \Sha^{d+1}(\tilde{C}) \to \Sha^{d+2}(\tilde{T}).  
\]

According to \cite{Izq16these}*{Chapitre~1, Remarque~3.8}, $\Sha^{d+2}(\tilde{T})$ is of cofinite type. We conclude that $\Sha^{d+1}(K,\tilde{C})$ is a finite group. Let $A$ denote the image of $\Sha^{d+1}_S(K,\tilde{C})$ in $\bigoplus_{v \in S} H^{d+1}(K_v, \tilde{C})$. Then $A$ has finite exponent as $\Sha_S^{d+1}(K,\tilde{C})$ has finite exponent. Since $H^{d+1}(K_v, \tilde{C})$ is of cofinite type according to \cite{Izq16these}*{Chapitre~1, Lemme~4.10(ii)}, $A$ is of cofinite type. Thus $A$ is a finite group. In conclusion, $\Sha^{d+1}_S(K,\tilde{C})$ is a finite group as desired. 
\end{proof}

The following theorem improves a special case of 
\cite{Izq16these}*{Chapitre 1, Corollarie 4.18}.

\begin{thm} [Global duality]\label{GlobalDuality}
    There is a perfect pairing of finite groups
    $$\Sha^1(K, C) \times \Sha^{d+1}(K, \tilde{C}) \to \mathbb{Q} / \mathbb{Z}.$$
\end{thm}
\begin{proof}
Since $\BX_{\ast}(\Tsc) \to \BX_{\ast}(T)$ is injective, by \cite{Izq16these}*{Chapitre 1, Corollarie 4.18}, we have the following perfect pairing of abelian groups:
\[\Sha^1(K,C) \times \overline{\Sha^{d+1}(K, \tilde{C})} \to 
\mathbb{Q} / \mathbb{Z},\]
where $\overline{\Sha^{d+1}(K, \tilde{C})}$ denotes the quotient of $\Sha^{d+1}(K, \tilde{C})$ by its maximal divisible subgroup. But by definition  $\Sha^{d+1}(K, \tilde{C}) \subseteq \Sha^{d+1}_S(K,\tilde{C})$, thus it is a finite group. Therefore, the maximal divisible subgroup of $\Sha^{d+1}(K,\tilde{C})$ is trivial. 
\end{proof}

% \begin{thm} \label{Dualities}  There are perfect pairings of abelian groups
% \begin{align}
% \mathbb{H}^0(K_v, C)^{\wedge} \times \mathbb{H}^{d+1}(K_v, \tilde{C}) &\to \mathbb{Q} / \mathbb{Z} \\
% \Sha^1(K, C) \times \overline{\Sha^{d+1}(K, \tilde{C})} &\to \mathbb{Q} / \mathbb{Z}
% \end{align}
% See \cite[Chapitre 1, Proposition 4.9 \& Corollarie 4.18]{Izq16these} for more details.
%\end{thm}

\section{The abelianization map}
Let $G$ be a connected reductive $K$-group
and let $T \subseteq G$ be a maximal torus. 
We write $T^{sc}$ for $\rho^{-1}(T) \subseteq G^{sc}$. Consider the complex of tori
$C = (T^{sc} \to T)$
where $T^{sc}$ is in degree $-1$ and $T$ is in degree $0$. According to \cite{Bor98}, we have the following abelianization map:
\[
\operatorname{ab}^0: H^0(K, G) \to H^0_{\rm{ab}}(K, G)  \coloneqq \mathbb{H}^0(K , C).
\]
We use this abelianization map to construct the third and fourth arrows in the sequence (\ref{Seq:FiniteMainTheorem}). 

\begin{para}\label{construction}
The morphism 
$
G(K_v) \to \Sha^{d+1}_S(K,\tilde{C})
$
is constructed as the composition of
\begin{equation} \label{seq:ConstuctionOfArrow}
    G(K_v) =  H^0(K_v, G) 
     \stackrel{\mrm{ab}^0}{\to}\mathbb{H}^0(K_v,C)
     \to \mathbb{H}^0(K_v,C)^{\wedge}
     \stackrel{\sim}{\to} \mathbb{H}^{d+1}(K_v,\tilde{C})^D 
     \to \Sha_S^{d+1}(K,\tilde{C})^D. 
\end{equation}
More explicitly, recall \cref{LocalDuality}, let $j_v$ denote the isomorphism $\mathbb{H}^{d+2}(K_v,\mathbb{Q}/\mathbb{Z}(d+1))\rightarrow\mathbb{Q}/\mathbb{Z}$, let $\hat{\operatorname{ab}}^0_v$ (resp. $\hat{\operatorname{ab}}^0$) denote the composition $G(K_v)\to\mathbb{H}^0(K_v,C)\to\mathbb{H}^0(K_v,C)^\wedge$ (resp. for $K$) and let $\operatorname{loc}_v$ denote the map $\mathbb{H}^r(K,\tilde{C})\rightarrow\mathbb{H}^r(K_v,\tilde{C})$. 
For any $c\in\Sha_S^{d+1}(K,\tilde{C})$ and $g\in G(K)$, the pairing of the image of $g$ in $\Sha_S^{d+1}(K,\tilde{C})^D$ is
\[(g,c)\colonequals\sum_{v\in S}j_v\left(\hat{\operatorname{ab}}^0_v(g)\cup \operatorname{loc}_v(c)\right)=\sum_{v\in X^{(1)}}j_v\left(\hat{\operatorname{ab}}^0_v(g)\cup \operatorname{loc}_v(c)\right).\]
\end{para}

\begin{prop}
The sequence below is a complex
    \begin{equation}\label{complex}
      1 \rightarrow \overline{G(K)}_S \rightarrow \prod_{v\in S}G(K_v) \rightarrow \Sha_S^{d+1}(K,\tilde{C})^D,
      \end{equation}
where $\overline{G(K)}_S$ is the closure of $G(K)$ in $\prod_{v\in S}G(K_v)$. 
\end{prop}

\begin{proof}
Taking $F=\mu_n^{\otimes (d+1)}$ in \cite{Izq16these}*{Chapitre~1, Th\'eor\`eme~2.7}, we obtain an exact sequence
$$H^{d+2}\left(K,\mu_n^{\otimes(d+1)}\right)\to\bigoplus_{v\in X^{(1)}}H^{d+2}\left(K_v,\mu_n^{\otimes(d+1)}\right)\to H^0\left(K,\mathbb{Z}/n\mathbb{Z}\right)^D\to 0.$$
Since $H^0\left(K,\mathbb{Z}/n\mathbb{Z}\right)=\mathbb{Z}/n\mathbb{Z}$, taking direct limit for all $n$ yields an exact sequence

\[0\to \mathbb{H}^{d+2}\left(K,\mathbb{Q}/\mathbb{Z}(d+1)\right)\to \bigoplus_{v\in X^{(1)}}\mathbb{H}^{d+2}\left(K_v,\mathbb{Q}/\mathbb{Z}(d+1)\right)\to \mathbb{Q}/\mathbb{Z}.\]
By Proposition \ref{prop:MotivicComplex}(ii) we conclude the exactness of
\begin{equation}\label{PT}
    0\to\mathbb{H}^{d+3}\left(K,\mathbb{Z}(d+1)\right)\xrightarrow{\operatorname{Loc}_v} \bigoplus_{v\in X^{(1)}}\mathbb{H}^{d+3}\left(K_v,\mathbb{Z}(d+1)\right)\xrightarrow{\sum j_v} \mathbb{Q}/\mathbb{Z}.
\end{equation}

By (\ref{construction}), for any $c\in\Sha_S^{d+1}(K,\tilde{C})$ and $g\in G(K)$, we have
\[
\begin{aligned}
    (g,c)=\sum_{v\in X^{(1)}}j_v\left(\hat{\operatorname{ab}}^0_v(g)\cup \operatorname{loc}_v(c)\right)=\sum_{v\in X^{(1)}}j_v\circ\operatorname{Loc}_v\left({\hat{\operatorname{ab}}}^0(g)\cup c\right)=0\ \ \text{(since (\ref{PT}) is exact)}. 
\end{aligned}
\]
Thus, by the continuity of $\overline{G(K)}_S \rightarrow \prod_{v\in S}G(K_v)$, (\ref{complex}) forms a complex. 
\end{proof}

\section{Proof of the main theorem}

In this section, we finish the proof of the main theorem.

%%%%%%%%%%%%%%%%%%%%%%%%%%%%%%%%%%%%%%%%%%%%%%
%%%%%%%%%% Lemma 1 证明 %%%%%%%%%%%%%%%%%%%%%%
%%%%%%%%%%%%%%%%%%%%%%%%%%%%%%%%%%%%%%%%%%%%%%

\begin{lem} \label{Lem1}
If the sequence \textup{(\ref{Seq:FiniteMainTheorem})} is exact for $G^m \times Q$,
where $m\ge 1$ is an integer and $Q$ is a quasi-trivial torus,
then it is exact for $G$.
\end{lem}
\begin{proof}
If (\ref{Seq:FiniteMainTheorem}) is exact for $G^m$, 
then it is exact for $G$ by construction. 
Now suppose that (\ref{Seq:FiniteMainTheorem}) is exact for $G \times_K Q$ for some quasi-trivial $K$-torus $Q$.
Since $T \subseteq G$ is a maximal torus, 
$T \times Q$ is a maximal torus of $G \times Q$. Moreover, the derived subgroup of $G\times Q$ is $\Gss$, so we have a composite $\rho_Q: \Gsc \to \Gss \to G \times Q$. We introduce the complex $C_Q = [\Tsc \to T \times Q]$ which is concentrated in degree -1 and 0. 

From the distinguished triangle
$\Tsc\to T\to C\to\Tsc[1]$,
we obtain the following commutative diagram with exact rows: 
\begin{center}
\begin{tikzcd}
0 \arrow{r} & H^1(K, T) \arrow{r} \arrow{d} & \mathbb{H}^1(K, C) \arrow{r} \arrow{d} & H^2(K, \Tsc) \arrow{d} \\
0 \arrow{r} & \displaystyle\prod_{v \in X^{(1)}}H^1(K_v,T) \arrow{r} & \displaystyle\prod_{v \in X^{(1)}}\mathbb{H}^1(K_v, C) \arrow{r} & \displaystyle\prod_{v \in X^{(1)}}H^2(K_v,\Tsc).
\end{tikzcd}
\end{center}
Hence, by taking kernel, there is an exact sequence 
$0\to\Sha^1(T)\to\Sha^1(C)\to\Sha^2(\Tsc)$. According to the functoriality of cohomology, we have the following commutative diagram:
\begin{center}
    \begin{tikzcd}
    0 \arrow{r} & \Sha^1(T \times Q) \arrow{r} \arrow[d, equal] & \Sha^1(C_Q) \arrow{r} \arrow{d} & \Sha^2(\Tsc) \arrow[d, equal] \\
    0 \arrow{r} & \Sha^1(T) \arrow{r} & \Sha^1(C) \arrow{r} & \Sha^2(\Tsc),
    \end{tikzcd}
\end{center}
where $\Sha^1(T \times Q) = \Sha(T)$ follows from $H^1(K,Q)=0$. A diagram chase yields the injectivity of $\Sha^1(C_Q) \to \Sha^1(C)$.

Similarly, using the distinguished triangles 
$\tilde{T}\times\tilde{Q}\to\widetilde{\Tsc}\to\tilde{C_Q}\rightarrow\tilde{T}[1]$
and    $\tilde{T}\to\widetilde{\Tsc}\to\tilde{C}\rightarrow\tilde{T}[1]$,
we obtain a commutative diagram with exact rows
\begin{center}
    \begin{tikzcd}
    0 \arrow[r] &
    \Sha^{d+1}_S(\tilde{C}) \arrow[r] \arrow[d] &
    \Sha^{d+2}_S(\tilde{T}) \arrow[d]\\
    0 \arrow[r] &
    \Sha^{d+1}_S(\widetilde{C_Q}) \arrow[r] &
    \Sha^{d+2}_S(\tilde{T}\times\tilde{Q}),
    \end{tikzcd}
\end{center}
where $\Sha_S^{d+2}(\tilde{T}) \to \Sha_S^{d+2}(\tilde{T} \times \tilde{Q})$ is injective. A diagram chase yields the injectivity of the map $\Sha^{d+1}_S(\tilde{C}) \to \Sha_S^{d+1}(\widetilde{C_Q})$.

Finally, recall that quasi-trivial $K$-tori are $K$-rational. Hence in particular $Q$ satisfies the weak approximation condition. 
It follows that the cokernel of the first map in (\ref{Seq:FiniteMainTheorem}) is stable under multiplying $G$ by a quasi-trivial torus. Let $A$ denote those two cokernels, then we have the following commutative diagram:
\begin{center}
    \begin{tikzcd}
    0 \arrow{r} & A \arrow{r} \arrow[d, equal] & \Sha^{d+1}_S(\widetilde{C_Q})^D \arrow{r} \arrow[d, two heads] & \Sha^1(C_Q) \arrow{r} \arrow[d, hook] & 0 \\
    0 \arrow{r} & A \arrow{r} & \Sha^{d+1}_S(\tilde{C})^D \arrow{r} & \Sha^1(C) \arrow{r} & 0,
    \end{tikzcd}
\end{center}
where the first row is exact by assumption, and the morphism $\Sha_S^{d+1}(\tilde{C})^D \to \Sha^1(C)$ is surjective since $\Sha^1(C) \cong \Sha^{d+1}(\tilde{C})^D$ by Theorem \ref{GlobalDuality}. A diagram chase yields that the sequence (\ref{Seq:FiniteMainTheorem}) is exact for $G$.
\end{proof}

Thus, by \cite{San81}*{Lemme~1.10}, we may assume that $G$ admits a special cover.

%%Construction of the map.
%%
\begin{lem}\label{Lemma2}
The groups $\prod_{v \in S}H^0(K_v,C)$ and $\prod_{v \in S}H^0(K_v,C)^{\wedge}$ have the same image in $\Sha^{d+1}_{S}(\tilde{C})^D$.
\end{lem}
\begin{proof}
By Lemma \ref{ShaFinite}, $\Sha_S^{d+1}(\tilde{C})^D$ is a finite group.
But the image of $H^0(K_v, C)$ is dense in that of $H^0(K_v,C)^{\wedge}$,
so they have the same image in $\Sha_S^{d+1}(\tilde{C})^D$.
\end{proof}

\begin{lem} \label{FirstExact}
For any finite subset $S \subseteq X^{(1)}$, the following sequence is exact:
$$
1 \rightarrow \overline{G(K)}_S \rightarrow \prod_{v\in S}G(K_v) \rightarrow \Sha_S^{d+1}(K,\tilde{C})^D.
$$
\end{lem}
%%%%%%%% Lemma 3证明%%%%%%%%%%%%%%%%%%%%%%%%%%%%%%%%%%%%%%%%%%
\begin{proof}
By Lemma $\ref{Lem1}$ , we assume that $G$ admits a special cover, i.e.,
\[
1\to F_0\to G_0\to G\to 1,
\]
where $F_0$ is finite and $G_0=G^{sc}\times Q$. Therefore we have the following commutative diagram whose upper rows are exact
\begin{equation} \label{DiagramInLemma3}
\begin{tikzcd} 
G_0(K)\arrow[r]\arrow[d]&G(K)\arrow[r, "\partial"]\arrow[d]&H^1(K,F_0)\arrow[r]\arrow[d]& H^1(K,G_0)\\
\displaystyle\prod_{v\in S}G_0(K_v)\arrow[r]&\displaystyle\prod_{v\in S}G(K_v)
\arrow[r]\arrow[d]&\displaystyle\prod_{v\in S}H^1(K_v,F_0)
\arrow[d]\\
&\Sha_S^{d+1}(K,\tilde{C})^D\arrow[r]&\Sha_S^{d+1}(K,\tilde{F_0})^D.
\end{tikzcd}
\end{equation}
Since $F_0$ is central in $G_0$, $F_0$ is contained in any maximal torus of $G_0$.
In particular, under hypotheses \ref{hypo:Sha2Vanish}, we can take a quasi-trivial maximal torus $T_0$ of $G_0$, then $F_0 \subseteq T_0$. Consequently the map $H^1(K,F_0) \to H^1(K,G_0)$ factors through $H^1(K,F_0) \to H^1(K,T_0)$. Since $T_0$ is a quasi-trivial torus, we have $H^1(K,T_0) = 0$ by Lemma \ref{Cohomology1VanishForQuasiTrivial}(i). Therefore, $H^1(K,F_0) \to H^1(K, G_0)$ is the trivial map, i.e., $\partial$ is surjective.

We show that the lower square in diagram (\ref{DiagramInLemma3}) commutes. Let $C_0 = [T_0 \to T]$ be the short complex associated to the isogeny $G_0 \to G$. Thus, we have a quasi-isomorphism $F_0[1] \simeq C_0$ of complexes. Note that we have a map $C \to C_0$ of complexes induced by the inclusion $\Gsc \to G$. Both $H^0(K_v,G) \to H^1(K_v,F_0) \simeq \mathbb{H}^0(K_v,G_0)$ and $H^0(K_v,G) \to \mathbb{H}^0(K_v,C) \to \mathbb{H}^0(K_v,C_0)$ are induced by the short exact sequence $1 \to F_0 \to G_0 \to G \to 1$. It follows that the left square in the following diagram commutes:
\begin{center}
    \begin{tikzcd}
    H^0(K_v,G) \arrow{r} \arrow{d} & \mathbb{H}^0(K_v,C) \arrow{r} \arrow{d} & \Sha_S^{d+1}(\tilde{C})^D \arrow{d} \\
    H^1(K_v,F_0) \arrow{r} &\mathbb{H}^0(K_v,C_0) \arrow{r} &  \Sha_S^{d+1}(\widetilde{C_0})^D.
    \end{tikzcd}
\end{center}
The right-hand side square commutes as a consequence of the functoriality of cup-product over $K_v$. Finally, taking the quasi-isomorphisms $F_0[1] \simeq C_0$ into account yields the commutativity of the lower square in diagram (\ref{DiagramInLemma3}).

Recall that $G_0$ satisfies weak approximation condition by Hypothesis \ref{hypo:GscWA} and the third column is exact by \cite{Izq16these}*{Chapitre~2, Lemme~2.2}. A diagram chase now yields the desired exact sequence.
\end{proof}
%%%%%%%%%%%%%%%%%%%%%%%%%%%%%%%%%%%%%%%%%%%%%%%%%%%%%%%%%%%%%
%%%%%%%%%%%%%%%%%%%%%%%%%%%%%%%%%%%%%%%%%%%%%%%%%%%%%%%%%%%%%%
\begin{thm}\label{MMainThm}
Let $G$ be a connected reductive group satisfying 
\textup{Hypothesis \ref{hypo:Sha2Vanish} and \ref{hypo:GscWA}}.
\begin{enumerate}
    \item [\textup{(i)}]Let $S\subset X^{(1)}$ be a finite set of 
    places. There is an exact sequence of groups
        \begin{equation} \label{Seq:FiniteMainTheoremInChapProof}
      1 \rightarrow \overline{G(K)}_S \rightarrow \prod_{v\in S}G(K_v) \rightarrow \Sha_S^{d+1}(K,\tilde{C})^D \rightarrow \Sha^1(K,C) \rightarrow 0.  
    \end{equation}
    Here $\overline{G(K)}_S$ denotes the closure of 
    the diagonal image of $G(K)$ in $\prod_{v\in S}
    G(K_v)$ for the product topology.
    \item [\textup{(ii)}]There is an exact sequence of groups
      \[
      1 \rightarrow \overline{G(K)} \rightarrow \prod_{v\in X^{(1)}}G(K_v) \rightarrow \Sha_{\omega}^{d+1}(K,\tilde{C})^D \rightarrow \Sha^1(K,C) \rightarrow 0.  
    \]
    Here $\overline{G(K)}$ denotes the closure 
    of the diagonal image of $G(K)$ in $\prod_{v\in X^{(1)}}
    G(K_v)$ for the product topology.
\end{enumerate}
\end{thm}
\begin{proof} 
    \hfill
    \begin{enumerate}%[leftmargin = 2pt]
        \item[(i)] Given Lemma \ref{FirstExact}, 
it only remains to show the following sequence 
$$\prod_{v \in S} G(K_v) \to \Sha^{d+1}_S(K,\tilde{C})^D \to \Sha^1(K,C) \to 0$$
is exact.

Recall (\ref{seq:ConstuctionOfArrow}) for the construction of the morphism $\prod_{v \in S}G(K_v) \to \Sha^{d+1}_S(K,\tilde{C})^D$.
By \cite{Tia21WA}*{Lemma 3.6}, the morphism $\mrm{ab}^0$ is surjective. Furthermore, 
by the perfect pair given by Theorem \ref{LocalDuality} and Lemma \ref{Lem1},
we can see that
$$\Im\left( \prod_{v \in S}G(K_v) \to \Sha^{d+1}_S(K,\tilde{C})^D \right)= \Im\left( \left(\prod_{v \in S}\mathbb{H}^{d+1}(K_v,\tilde{C})  \right)^D \to \Sha^{d+1}_S (K,\tilde{C})^D \right). $$

From the definition of Tate--Shafarevich groups,
there is an exact sequence $$0 \to \Sha^{d+1}(K,\tilde{C}) \to \Sha_S^{d+1}(K, \tilde{C}) \to \prod_{v \in S} \mathbb{H}^{d+1}(K_v, \tilde{C}).$$
By taking dualities, 
there is an exact sequence
$$\left( \prod_{v \in S}\mathbb{H}^{d+1}(K_v,\tilde{C}) \right)^{D} \to \Sha_S^{d+1}(K,\tilde{C})^D \to \Sha^{d+1}(K,\tilde{C})^{D} \to 0.$$
Here by \cref{GlobalDuality} we have 
$$\Sha^{d+1}(K,\tilde{C})^D \stackrel{\sim}{\to} \Sha^1(K,C). $$
Therefore, we have established the exact sequence
$$\prod_{v \in S} G(K_v) \to \Sha^{d+1}_S(K,\tilde{C})^D \to \Sha^1(K, C) \to 0. $$
    
    \item[(ii)] Passing to projective limit in (\ref{Seq:FiniteMainTheoremInChapProof}) over all finite subset $S \subset X^{(1)}$, we obtain an exact sequence
    $$1 \rightarrow \overline{G(K)} \rightarrow \prod_{v\in X^{(1)}}G(K_v) \rightarrow \Sha_\omega^{d+1}(K,\tilde{C})^D. $$
    Dualizing the exact sequence of abelian groups
    \[
    1 \to \Sha^{d+1}(K, \tilde{C}) \to \Sha^{d+1}_{\omega}(K, \tilde{C}) \to \bigoplus_{v \in X^{(1)}} \mathbb{H}^{d+1}(K_v,\tilde{C})
    \]
    yields an exact sequence of abelian groups via local and global dualities (\cref{LocalDuality} and \cref{GlobalDuality}):
    \[
    \prod_{x \in X^{(1)}}\mathbb{H}^{0}(K_v,C)^{\wedge} \to \Sha^{d+1}_{\omega}(K, \tilde{C})^D \to \Sha^{1}(K, C)^D \to 0.
    \]
    Since $G(K_v) \to \mathbb{H}^0(K_v,C)$ is surjective and $\prod_{v \in X^{(1)}}\mathbb{H}^0(K_v,C)$ is dense in $\prod_{v \in X^{(1)}}\mathbb{H}^0(K_v,C)$, it will be sufficient to show the image of $\prod_{v \in X^{(1)}}G(K_v)$ is closed in $\Sha^{d+1}_{\omega}(\tilde{C})^D$. Recall that $H^1(K, T_0) = H^1(K_v, T_0) = 0$ by Lemma \ref{Cohomology1VanishForQuasiTrivial},
    in view of diagram (\ref{DiagramInLemma3})
    
\begin{center}
\begin{tikzcd} 
G_0(K)\arrow[r]\arrow[d]&G(K)\arrow[r]\arrow[d]&H^1(K,F_0)\arrow[r]\arrow[d]& 0\\
\displaystyle\prod_{v \in X^{(1)}}G_0(K_v)\arrow[r]&\displaystyle\prod_{v\in X^{(1)}}G(K_v)
\arrow[r]&\displaystyle\prod_{v\in X^{(1)}}H^1(K_v,F_0)
 \arrow[r] & 0.
\end{tikzcd}
\end{center}
By a diagram chase, we obtain the quotient of $\prod_{v \in X^{(1)}}G(K_v)$ by $\overline{G(K)}$ is compact since it is isomorphic to the quotient of profinite group $\prod_{v \in X^{(1)}}H^1(K_v,F_0)$ by $\overline{H^1(K,F_0)}$ (which means its closure in the former group with respect to the profinite topology). Since the group $\Sha^{d+1}(K,\tilde{C})^D$ is profinite and thus Hausdorff, the image of $\prod_{v \in X^{(1)}}G(K_v)$ in $\Sha^{d+1}(K,\tilde{C})^D$ is closed, which completes the proof as desired.
\qedhere
\end{enumerate}
\end{proof}

\vspace{15pt}

\begin{center}
\begin{tabular}{l@{\hspace{8em}}l} 
   \textbf{Zhongda LI}, \textbf{Haoxiang PAN} & \textbf{Che LIU}   \\ 
   School of Mathematical Science & School of Mathematical Science \\ 
   Peking University & Fudan University  \\ 
   Beijing 100000, China & Shanghai 200433, China  \\
   Email: stbylyz@stu.pku.edu.cn & Email: cheliu22@m.fudan.edu.cn \\
   \quad \quad ~ ~ p135146725@stu.pku.edu.cn
\end{tabular}
\end{center}
\end{document}